\documentclass[11pt]{amsart}
\usepackage{amsfonts,amssymb,amsthm,eucal,color,bbm}

\usepackage[margin=1.25in]{geometry}

\newtheorem{thm}{Theorem}
\newtheorem{cor}[thm]{Corollary}
\newtheorem{lemma}[thm]{Lemma}
\newtheorem{prop}[thm]{Proposition}

\newcommand{\R}{\mathbb{R}}
\newcommand{\E}{\mathbb{E}}

\newcommand{\Prob}{\mathbb{P}}
\newcommand{\N}{\mathbb{N}}
\newcommand{\Z}{\mathbb{Z}}
\newcommand{\C}{\mathbb{C}}

\DeclareMathOperator{\Var}{Var}
\DeclareMathOperator{\Cov}{Cov}
\newcommand{\inprod}[2]{\left\langle #1, #2 \right\rangle}
\renewcommand{\Re}{\operatorname{Re}}
\renewcommand{\Im}{\operatorname{Im}}
\newcommand{\abs}[1]{\left\vert #1 \right\vert}
\newcommand{\norm}[1]{\left\Vert #1 \right\Vert}

\newcommand{\eps}{\varepsilon}

\newcommand{\Circle}{\mathbb{T}}
\newcommand{\Normal}{\mathcal{N}}
\newcommand{\ind}[1]{\mathbbm{1}_{#1}}

\allowdisplaybreaks[3]
\numberwithin{thm}{section}
\numberwithin{equation}{section}

\author{Mark W.\ Meckes}

\address{Department of Mathematics, Case Western Reserve University,
10900 Euclid Ave., Cleveland, Ohio 44106, U.S.A.}

\email{mark.meckes@case.edu}
\urladdr{www.case.edu/artsci/math/mwmeckes/}

\title{The spectra of random abelian $G$-circulant matrices}

\subjclass[2010]{Primary 60B20; secondary 15B99, 43A25, 60F05}

\keywords{Random matrix, $G$-circulant matrix, spectral distribution}

\begin{document}

\maketitle

\begin{abstract}
  This paper studies the asymptotic behavior of eigenvalues of random
  abelian $G$-circulant matrices, that is, matrices whose structure is
  related to a finite abelian group $G$ in a way that naturally
  generalizes the relationship between circulant matrices and cyclic
  groups.  It is shown that, under mild conditions, when the size of
  the group $G$ goes to infinity, the spectral measures of such random
  matrices approach a deterministic limit.  Depending on some aspects
  of the structure of the groups, whether the matrices are constrained
  to be Hermitian, and a few details of the distributions of the
  matrix entries, the limit measure is either a (complex or real)
  Gaussian distribution or a mixture of two Gaussian distributions.
\end{abstract}



\section{Introduction}

Given a finite group $G$ and a function $f : G \to \C$, the matrix $M
= \bigl[f(ab^{-1})\bigr]_{a,b \in G}$ is called a $G$-circulant matrix
by Diaconis \cite{Diaconis-book,Diaconis-matrices}.  This generalizes
the classical notion of circulant matrices, which arise as the special
case in which $G$ is a finite cyclic group.  The action of such a
matrix $M$ on the vector space $\{ g : G \to \C\}$ is as a convolution
operator: for $g : G \to \C$ and $a \in G$,
\begin{equation}\label{E:convolution}
(Mg)(a) = \sum_{b \in G} f(ab^{-1}) g(b) =: (f*g)(a).
\end{equation}

This paper considers the asymptotic behavior of the spectra of random
$G$-circulant matrices, or equivalently random convolution operators
on $G$, when $G$ is a large abelian group.  (For the rest of this
paper, $G$ will always stand for a \emph{finite abelian} group.)  Such
random matrices will be generated by picking the values $f(a)$
independently, with or without imposing a constraint $f(a^{-1}) =
\overline{f(a)}$ which is equivalent to insisting that the matrix $M$
is Hermitian.  This generalizes the study of random circulant
matrices, whose theory has already been developed in
\cite{BoMi,BoSe,BrSe,Meckes,BoHaSa} among many other papers, with
applications discussed in \cite{JaSr,YiMoYaZh}.  The richer structure
of arbitrary abelian groups relative to cyclic groups leads to the
appearance of some interesting phenomena which do not occur for
circulant matrices, or the more familiar setting of random matrices
with independent entries.

The prototypical situation (exemplified in Corollaries
\ref{T:C-Ginibre-limit}, \ref{T:GUE-limit}, and \ref{T:Z2-GOE-limit},
and Theorems \ref{T:circular-law-uncorrelated} and
\ref{T:semicircle-law-special} below) is that when the size of $G$
grows the empirical spectral distribution of a (properly normalized)
random $G$-circulant matrix $M$ approaches a Gaussian distribution.
When $M$ is constrained to be Hermitian the limit will be a real
Gaussian distribution; without such a constraint it will be a complex
Gaussian distribution.  These situations may be thought of as
analogous to the semicircle law for Hermitian random matrices and
circular law for non-Hermitian random matrices with independent
entries, respectively.  This behavior, which has already been observed
for random circulant matrices in \cite{BoMi,Meckes}, occurs in
particular if only a negligible fraction of the elements of $G$ are of
order $2$, and also if every nonidentity element of $G$ is of order
$2$.  On the other hand, if neither of these is the case then more
complicated limiting distributions occur which are mixtures of two
Gaussian distributions (as in Theorems \ref{T:circular-law-correlated}
and \ref{T:semicircle-law-general} below).

Another perspective on these results, which is crucial in the proofs,
is that they describe the distribution of values of random Fourier
series on $G$.  The supremum of such a random Fourier series is
already a thoroughly studied quantity \cite{Kahane,MaPi}.  In
particular, results of Marcus and Pisier \cite{MaPi} include as
special cases estimates of the spectral norms of random $G$-circulant
matrices, as pointed out in Proposition \ref{T:norm} below.

Section \ref{S:Fourier} below briefly reviews the facts about Fourier
analysis on finite abelian groups which are used here and points out
their immediate consequences for $G$-circulant matrices; some notation
and conventions used in the remainder of the paper are established
there.  Section \ref{S:Gaussian} investigates the spectra of some
random $G$-circulant matrices whose entries are Gaussian random
variables. The invariance properties of Gaussian random variables
allow an easy detailed study to be undertaken which illuminates the
general situation, in particular the role of the number of elements of
order $2$.  Finally, Section \ref{S:general} determines the asymptotic
behavior of the spectrum for general entries with finite variances.

The cases of $G$-circulant matrices with heavy-tailed entries, and of
random $G$-circulant matrices when $G$ is a nonabelian finite group,
will be investigated in future work.

\section*{Acknowledgements}

The author thanks Persi Diaconis for encouragement and pointers to the
literature, John Duncan for helpful discussions about character
theory, and the referee for careful reading and useful comments.  This
research was partly supported by National Science Foundation grant
DMS-0902203.


\section{Some Fourier analysis and notation}
\label{S:Fourier}

For a finite abelian group $G$, we denote by $\widehat{G}$ the family
of group homomorphisms $\chi : G \to \Circle$, where $\Circle$ is the
multiplicative group $\{z \in \C \mid \abs{z} = 1\}$. The elements of
$\widehat{G}$ are called characters of $G$; $\widehat{G}$ is a group
under the operation of pointwise multiplication. The multiplicative
inverse of a character $\chi$ is its pointwise complex conjugate
$\overline{\chi}$.  From the homomorphism property it follows that for
$a \in G$ and $\chi \in \widehat{G}$, $\chi(a^{-1}) =
\overline{\chi}(a)$.

We denote by $\ell^2(G)$ the space of functions $f: G \to \C$ equipped
with the inner product
\[
\inprod{f}{g} = \sum_{a \in G} f(a) \overline{g(a)},
\]
and $\ell^2(\widehat{G})$ is defined analogously.  The Fourier
transform of $f \in \ell^2(G)$ is the function $\widehat{f} \in
\ell^2(\widehat{G})$ given by
\[
\widehat{f}(\chi) = \inprod{f}{\overline{\chi}}
= \sum_{a \in G} f(a) \chi(a).
\]
This includes as special cases both the classical discrete Fourier
transform (when $G$ is cyclic) and the Walsh--Hadamard transform (when
$G$ is a product of cyclic groups of order $2$). The following lemma
summarizes the most important fundamental facts about the Fourier
transform for our purposes.

\begin{lemma} \label{T:FT-isometry}
  Let $G$ be a finite abelian group with $\abs{G}$ elements.
  \begin{enumerate}
  \item \label{I:onb} The functions $\bigl\{
    \frac{1}{\sqrt{\abs{G}}}\chi \mid \chi \in \widehat{G}\bigr\}$
    form an orthonormal basis of $\ell^2(G)$.
  \item \label{I:isometry} The map $f \mapsto
    \frac{1}{\sqrt{\abs{G}}}\widehat{f}$ is a linear isometry of
    $\ell^2(G)$ onto $\ell^2(\widehat{G})$.
  \item \label{I:convolution} If $f, g \in \ell^2(G)$, then for each
    $\chi \in \widehat{G}$, $\widehat{f*g}(\chi) = \widehat{f}(\chi)
    \widehat{g}(\chi)$ (where the convolution $f*g$ is defined in
    \eqref{E:convolution}.
  \end{enumerate}
\end{lemma}

\begin{proof}
  \begin{enumerate}
  \item See Theorem 6 on \cite[p.\ 19]{Serre}.
  \item This follows easily from Proposition 7 on \cite[p.\ 20]{Serre}
    (which is a consequence of part (\ref{I:onb})).
  \item This follows directly from the definitions by a
    straightforward computation. \qedhere
  \end{enumerate}
\end{proof}

\medskip

Observe that contained in Lemma \ref{T:FT-isometry}(\ref{I:onb}) is
the fact that $\abs{G} = \bigl\vert\widehat{G}\bigr\vert$.

\medskip

We will need two additional facts about characters of finite abelian
groups which are not as easily located in standard references.

\begin{lemma}\label{T:p2}
  The number of elements $a \in G$ such that $a^2 = 1$ is equal to the
  number of characters $\chi \in \widehat{G}$ such that $\chi =
  \overline{\chi}$.
\end{lemma}

\begin{proof}
  For $a \in G$, define $\delta_a : G \to \C$ by $\delta_a(b) =
  \delta_{a,b}$, where the latter is the Kronecker delta function, and
  observe that $\{ \delta_a \mid a \in G\}$ is an orthonormal basis of
  $\ell^2(G)$. Then $\widehat{\delta_a}(\chi) = \chi(a)$ for each
  $\chi \in \widehat{G}$. By Lemma
  \ref{T:FT-isometry}(\ref{I:isometry}), the number of $a \in G$ such
  that $a^2 = 1$ is equal to
  \begin{align*}
    \sum_{a \in G} \inprod{\delta_a}{\delta_{a^{-1}}}
    &= \frac{1}{\abs{G}} \sum_{a \in G} 
    \inprod{\widehat{\delta_a}}{\widehat{\delta_{a^{-1}}}} 
    = \frac{1}{\abs{G}} \sum_{a \in G} \sum_{\chi \in \widehat{G}}
    \chi(a) \overline{\chi}(a^{-1}) \\
    &= \frac{1}{\abs{G}} \sum_{\chi \in \widehat{G}} \sum_{a \in G} 
    \chi(a)^2 
    = \frac{1}{\abs{G}} \sum_{\chi \in \widehat{G}}
    \inprod{\chi}{\overline{\chi}},
  \end{align*}
  which by Lemma \ref{T:FT-isometry}(\ref{I:onb}) is equal to the
  number of $\chi \in \widehat{G}$ such that $\chi = \overline{\chi}$.
\end{proof}

\medskip

Lemma \ref{T:p2} says that $G$ and $\widehat{G}$ have equal numbers of
elements of order $2$.  A much stronger fact is also true: $G$ and
$\widehat{G}$ are isomorphic groups.  However, this isomorphism is
noncanonical, depends on the classification of finite abelian groups,
and in any case is not useful here.

\begin{lemma}\label{T:extensions}
  Let $H$ be a subgroup of a finite abelian group $H$. Then each
  character on $H$ extends to a character on $G$ in precisely
  $\abs{G}/\abs{H}$ distinct ways.
\end{lemma}

\begin{proof}
  It is easy to check that restriction to $H$ defines a homomorphism
  $\widehat{G} \to \widehat{H}$. Since each coset of this
  homomorphism's kernel has the same size, it suffices to prove that
  that it is surjective, or equivalently that each character on $H$
  extends to a character on $G$ at all.  For a proof of this fact see,
  e.g., \cite[p.\ 134]{Apostol}.
\end{proof}

\medskip

From \eqref{E:convolution} and Lemma
\ref{T:FT-isometry}(\ref{I:convolution}) it follows that the Fourier
transform diagonalizes $G$-circulant matrices.  In particular, if $M =
[f(ab^{-1})]_{a,b \in G}$ for $f \in \ell^2(G)$, then the eigenvalues
of $M$ are precisely the values $\bigl\{\widehat{f}(\chi) \mid \chi
\in \widehat{G}\bigr\}$ of the Fourier transform of $f$, and the
characters of $G$ are eigenvectors of $M$.  (For generalizations of
these facts for nonabelian $G$, see
\cite{Diaconis-book,Diaconis-matrices}.) Observe that every
$G$-circulant matrix is normal, but that $M$ is Hermitian if and only
if $f(a^{-1}) = \overline{f(a)}$ for each $a \in G$.

Given a family of random variables $\{ Y_a \mid a \in G\}$, define the
random function $f \in \ell^2(G)$ by $f(a) = \frac{1}{\sqrt{\abs{G}}}
Y_a$.  (We are avoiding using $X$ to name random variables because of
its typographical similarity to $\chi$.)  The corresponding
$G$-circulant matrix is the random matrix $M = \bigl[ Y_{ab^{-1}}
\bigr]_{a,b \in G}$.  Its eigenvalues, indexed by $\chi \in
\widehat{G}$, are given by
\begin{equation}\label{E:eigenvalue-formula}
  \lambda_\chi = \widehat{f}(\chi) = \frac{1}{\sqrt{\abs{G}}} 
  \sum_{a \in G} Y_a \chi(a),
\end{equation}
and the empirical spectral distribution of $M$ is
\[
\mu = \frac{1}{\bigl\vert \widehat{G} \bigr\vert} \sum_{\chi \in \widehat{G}}
\delta_{\lambda_\chi}
= \frac{1}{\abs{G}} \sum_{\chi \in \widehat{G}}
\delta_{\lambda_\chi},
\]
where $\delta_z$ here denotes the point mass at $z \in \C$.

The Fourier transform $\widehat{f}$ is a random trigonometric
polynomial on $G$, of the kind studied extensively by Marcus and
Pisier \cite{MaPi}.  From \eqref{E:eigenvalue-formula} it follows in
particular that $\norm{M} = \bigl\Vert \widehat{f} \bigr\Vert_\infty$,
where the former norm is the spectral norm of $M$.  The following
result is thus a special case of \cite[Theorem 1.4]{MaPi}, which also
applies to infinite compact abelian groups.

\begin{prop}\label{T:norm}
  Suppose that $\{Y_a \mid a \in G\}$ are independent (except possibly
  for a constraint $Y_{a^{-1}} = \overline{Y_a}$ for each $a \in G$)
  and mean $0$ with finite second moments. Then
  \[
  c \left(\min_{a \in G} \E \abs{Y_a}\right) \le 
  \frac{\E \norm{M}}{\sqrt{\log\abs{G}}} \le 
  C \sqrt{\max_{a \in G} \E \abs{Y_a}^2},
  \]
  where $c, C > 0$ are constants, independent of $G$ and the
  distributions of the $Y_a$.
\end{prop}

\medskip

The rest of this paper deals mainly with infinite sequences of finite
abelian groups $G^{(n)}$, always assumed to satisfy $\abs{G^{(n)}} \to
\infty$.  For each $n$ a family of random variables $\bigl\{Y_g^{(n)}
\mid g \in G^{(n)} \bigr\}$ will be used to construct a random
$G^{(n)}$-circulant matrix
\[
M^{(n)}= \left[\frac{1}{\sqrt{\abs{G^{(n)}}}} Y^{(n)}_{ab^{-1}}
  \right]_{a,b \in G^{(n)}}
\]
with empirical spectral measure $\mu^{(n)}$.  As mentioned earlier,
an important role will be played by the quantity
\[
p_2^{(n)} = \frac{\abs{\{ a \in G^{(n)} \mid a^2 = 1\}}}{\abs{G^{(n)}}}
=
\frac{\bigl\vert \bigl\{\chi \in \widehat{G} \mid \chi = \overline{\chi} \bigr
  \}\bigr\vert }
  {\bigl\vert\widehat{G}^{(n)}\bigr\vert}.
\]

The standard real Gaussian measure is denoted $\gamma_\R$, and the
standard complex Gaussian distribution, normalized such that $\E
\abs{Z}^2 = 1$ when $Z$ is a standard complex Gaussian random
variable, is denoted $\gamma_\C$. For $\alpha \in [0,1]$,
$\gamma_\alpha$ denotes the Gaussian measure on $\C \cong \R^2$ with
covariance
$\frac{1}{2} \bigl[ \begin{smallmatrix} 1+\alpha & 0 \\
  0 & 1-\alpha \end{smallmatrix}\bigr]$, so that in particular
$\gamma_0 = \gamma_\C$ and $\gamma_1 = \gamma_\R$.

The integral of a function $f$ with respect to a measure $\nu$ will be
denoted by $\nu(f)$.


\section{Gaussian matrix entries} \label{S:Gaussian}

The following is an immediate consequence of Lemma
\ref{T:FT-isometry}(\ref{I:isometry}) and the rotation-invariance of
the standard Gaussian distribution.  The special case of this result
for classical circulant matrices (that is, when $G$ is a cyclic group)
was observed in \cite{Meckes}.

\begin{prop}\label{T:C-Ginibre-eigenvalues}
  Let $G$ be a finite abelian group and let $\{Y_a \mid a \in G\}$ be
  independent, standard complex Gaussian random variables. Then the
  eigenvalues $\bigl\{ \lambda_\chi \mid \chi \in \widehat{G} \bigr\}$
  of $M$ given by \eqref{E:eigenvalue-formula} are independent,
  standard complex Gaussian random variables.
\end{prop}

\medskip

The random matrix ensemble in Proposition
\ref{T:C-Ginibre-eigenvalues} is the $G$-circulant analogue of the
complex Ginibre ensemble $X$, which consists of a square matrix with
independent, standard complex Gaussian entries.

\begin{cor}\label{T:C-Ginibre-limit}
  Suppose that for each $n$, $\bigl\{Y_a^{(n)} \mid a \in G^{(n)}
  \bigr\}$ are independent, standard complex Gaussian random
  variables. Then $\E \mu^{(n)} = \gamma_\C$ for each $n$, and
  $\mu^{(n)} \to \gamma_\C$ weakly in probability.  Furthermore, if
  $\abs{G^{(n)}} = \Omega(n^{\eps})$ for some $\eps > 0$, then
  $\mu^{(n)} \to \gamma_\C$ weakly almost surely.
\end{cor}

\begin{proof}
  For each, say, Lipschitz $f: \C \to \R$,
  \[
  (\E\mu)(f) := \E \bigl(\mu(f)\bigr) 
  = \frac{1}{\abs{G}} \sum_{\chi \in \widehat{G}} \E f(\lambda_\chi),
  \]
  where the $(n)$ superscripts are omitted for simplicity.  By
  Proposition \ref{T:C-Ginibre-eigenvalues}, each $\lambda_\chi$ is
  distributed according to $\gamma_\C$, and so \( (\E\mu)(f) =
  \gamma_\C(f) \). Thus $\E \mu = \gamma_\C$.

  By the concentration properties of Gaussian measure (see
  \cite{Ledoux}), since the $\lambda_\chi$ are distributed as
  independent standard complex Gaussian random variables, if $f$ is
  $1$-Lipschitz, then
  \[
  \Prob \bigl[\abs{\mu(f) - \gamma_\C(f)} \ge t \bigr] 
  \le 2 e^{-\abs{G} t^2}
  \]
  for each $t > 0$. If $\abs{G^{(n)}} = \Omega(n^{\eps})$, then the
  Borel--Cantelli lemma implies that $\mu^{(n)}(f) \to \gamma_\C(f)$
  almost surely.  Applying this to a countable dense family of $f$, it
  follows that $\mu^{(n)} \to \gamma_\C$ weakly almost surely.
  
  In the general case, since $\abs{G^{(n)}} \to \infty$, each
  subsequence of $\mu^{(n)}$ has a subsequence $\mu^{(n_j)}$ for
  which, say, $\abs{G^{(n_j)}} \ge j$, so that by the above argument
  $\mu^{(n_j)}$ converges to $\gamma_\C$ almost surely as $j \to
  \infty$.  It follows that $\mu^{(n)}$ converges to $\gamma_\C$ in
  probability.
\end{proof}

\medskip

The next proposition deals with the $G$-circulant analogue of the
Gaussian Unitary Ensemble (GUE), which, up to a choice of
normalization, is distributed as $2^{-1/2}(X + X^*)$, where $X$ is the
complex Ginibre ensemble mentioned above.  Equivalently, the diagonal
entries of the GUE are standard real Gaussian random variables, the
off-diagonal entries are standard complex Gaussian random variables,
and the entries are independent except for the constraint that the
matrix is Hermitian.  It is worth noting explicitly that while each
entry of the GUE has (complex) variance 1, the variance of a diagonal
entry and the real part of an off-diagonal entry differ by a factor of
$2$.  (Again, the special case for classical circulant matrices was
observed earlier in \cite{Meckes}.)

\begin{prop}\label{T:GUE-eigenvalues}
  Let $G$ be a finite abelian group and let $\{Y_a \mid a \in G \}$ be
  random variables which are independent except for the constraint
  $Y_{a^{-1}} = \overline{Y_a}$, and such that
  \[
  Y_a \sim \begin{cases} \gamma_\R & \text{ if } a^2 = 1,\\
    \gamma_\C & \text{ if } a^2 \neq 1.
  \end{cases}
  \]
  Then the eigenvalues $\bigl\{ \lambda_\chi \mid \chi \in \widehat{G}
  \bigr\}$ of $M$ given by \eqref{E:eigenvalue-formula} are
  independent, standard real Gaussian random variables.
\end{prop}

\begin{proof}
  Let $\{Z_a \mid a \in G\}$ be independent, standard complex Gaussian
  random variables. Then $\{Y_a \mid a \in G\}$ are distributed as
  $\bigl\{ 2^{-1/2}\bigl(Z_a + \overline{Z_{a^{-1}}}\bigr) \mid a \in
  G\bigr\}$. Thus the eigenvalues $\lambda_\chi$ of $M$ in the present
  proposition are jointly distributed as $\sqrt{2}$ times the real
  parts of the eigenvalues of the random matrix defined in Proposition
  \ref{T:C-Ginibre-eigenvalues}, and are thus independent real
  standard normal random variables.
\end{proof}

\medskip

Observe that in the ``$G$-circulant GUE'' of Proposition
\ref{T:GUE-eigenvalues}, every element $a \in G$ with $a = a^{-1}$
corresponds to a ``diagonal'' of $M$ in which the entries are
constrained to be real.

The following corollary follows from Proposition
\ref{T:GUE-eigenvalues} in the same way that Corollary
\ref{T:C-Ginibre-limit} follows from Proposition
\ref{T:C-Ginibre-eigenvalues}.

\begin{cor}\label{T:GUE-limit}
  Suppose that for each $n$, $\bigl\{Y_a^{(n)} \mid a \in G^{(n)}
  \bigr\}$ are real and complex Gaussian random variables as described
  in Proposition \ref{T:GUE-eigenvalues}. Then $\E \mu^{(n)} =
  \gamma_\R$ for each $n$, and $\mu^{(n)} \to \gamma_\R$ weakly in
  probability.  Furthermore, if $\abs{G^{(n)}} = \Omega(n^{\eps})$ for
  some $\eps > 0$, then $\mu^{(n)} \to \gamma_\R$ weakly almost
  surely.
\end{cor}

\medskip

The real Ginibre ensemble $X$ consists of a square matrix with
independent, real standard Gaussian random variables.  The Gaussian
Orthogonal Ensemble (GOE) is distributed as $2^{-1/2}(X+X^t)$.
Equivalently, the diagonal entries of the GOE are distributed as
$\Normal(0,2)$ and the off-diagonal entries are distributed as
$\Normal(0,1)$.  In general the analogues of Propositions
\ref{T:C-Ginibre-eigenvalues} and \ref{T:GUE-eigenvalues} for matrices
with real entries are less elegant. In the nonsymmetric case the
eigenvalues have a Gaussian joint distribution in a
$\abs{G}$-dimensional real subspace of $\C^{\abs{G}}$, and in the
symmetric case the $\abs{G}$ eigenvalues are not independent in
general.  We will not state such results in general, but will note for
future reference that in the ``$G$-circulant GOE'', every element $a
\in G$ with $a = a^{-1}$ corresponds to a diagonal of $M$ in which the
variance of the entries is $2$ instead of $1$.  (See Theorem
\ref{T:semicircle-law-general} below and the discussion following it.)

On the other hand, the analogous results \emph{are} simple in the case
in which the characters $\chi \in \widehat{G}$ are all real-valued, so
that the Fourier transform defines an isometry (up to scaling) between
the \emph{real} $\ell^2$ spaces on $G$ and $\widehat{G}$.  By Lemma
\ref{T:p2}, this is the case precisely when every $a \in G$ satisfies
$a^2 = 1$, or in other words, when $G \cong (\Z_2)^n$ for some $n$.
In this case a $G$-circulant matrix is automatically symmetric, so
that there is no difference (except for scaling) between the
``$G$-circulant real Ginibre ensemble'' and the ``$G$-circulant GOE''.
The following results are proved in the same way as Proposition
\ref{T:C-Ginibre-eigenvalues} and Corollary \ref{T:C-Ginibre-limit}.

\begin{prop}\label{T:Z2-GOE-eigenvalues}
  Let $G \cong (\Z_2)^n$ and let $\{Y_a \mid a \in G\}$ be
  independent, standard real Gaussian random variables. Then the
  eigenvalues $\bigl\{ \lambda_\chi \mid \chi \in \widehat{G} \bigr\}$
  of $M$ given by \eqref{E:eigenvalue-formula} are independent,
  standard real Gaussian random variables.
\end{prop}

\begin{cor}\label{T:Z2-GOE-limit}
  Suppose that for each $n$, $G^{(n)} \cong (\Z_2)^n$ and $\{Y_a^{(n)}
  \mid a \in G^{(n)}\}$ are independent, standard real Gaussian random
  variables. Then $\E \mu^{(n)} = \gamma_\R$ for each $n$, and
  $\mu^{(n)} \to \gamma_\R$ weakly almost surely.
\end{cor}


\section{General matrix entries} \label{S:general}

Our main results are stated under a Lindeberg-type condition on the
random variables $Y_a^{(n)}$ used to generate the random matrices:
\begin{equation}\label{E:Lindeberg-condition}
  \forall \eps > 0 : \quad 
  \lim_{n \to \infty} \frac{1}{\abs{G^{(n)}}} 
  \sum_{a\in G^{(n)}} \E \Bigl( \bigl\vert Y_a^{(n)}\bigr\vert^2 
  \ind{\vert Y_a^{(n)}\vert 
    \ge \eps \sqrt{\vert G^{(n)} \vert }} \Bigr) 
  = 0.
\end{equation}
The usual remarks apply about the sufficiency of identical
distribution or a Lyapunov-type condition:
\eqref{E:Lindeberg-condition} holds in the settings of Theorems
\ref{T:circular-law-correlated} and \ref{T:circular-law-uncorrelated}
if all the $Y_a^{(n)}$ are identically distributed, or have uniformly
bounded $(2 + \delta)$ moments; it holds in the settings of Theorems
\ref{T:semicircle-law-general} and \ref{T:semicircle-law-special} if
all the random variables with a given variance assumption satisfy such
assumptions.

We now state our main results, deferring the proofs until the end of
the section.

\begin{thm}\label{T:circular-law-correlated}
  Let $\alpha \in [0,1]$. Suppose that for each $n$, $\{Y_a^{(n)} \mid
  a \in G^{(n)}\}$ are independent; that
  \[
  \E Y_a^{(n)} = 0, \quad \E \bigl\vert Y_a^{(n)}\bigr\vert^2 = 1, \quad
  \text{and} \quad \E \bigl(Y_a^{(n)}\bigr)^2 = \alpha
  \]
  for every $a \in G^{(n)}$; and that \eqref{E:Lindeberg-condition}
  holds. Suppose further that $\lim_{n\to \infty} p_2^{(n)} = p$
  exists.  Then $\mu^{(n)}$ converges, in mean and in probability, to
  $(1 - p) \gamma_\C + p \gamma_\alpha$.
\end{thm}

\medskip

One of the main special cases of interest in Theorem
\ref{T:circular-law-correlated} is when $\alpha = 1$, that is, when
the matrix entries are all real. In that case, the limiting spectral
distribution of $M^{(n)}$ is complex Gaussian if the number of $a$
with $a^2 = 1$ is negligible for large $n$. On the other hand, if
the fraction of such $a$ is asymptotically constant then, due to the
presence of many real-valued characters $\chi$, the limiting spectral
distribution will be a mixture of $\gamma_\C$ and $\gamma_\R$.

The other main special case of interest is when $\alpha = 0$, so that
the matrix entries have uncorrelated real and imaginary parts.  In
that case, which generalizes the setting of Corollary
\ref{T:C-Ginibre-limit}, one can remove the assumption that
$p_2^{(n)}$ approaches a limit.

\begin{thm}\label{T:circular-law-uncorrelated}
  Suppose that for each $n$, $\{Y_a^{(n)} \mid a \in G^{(n)}\}$ are
  independent; that
  \[
  \E Y_a^{(n)} = 0, \quad \E \bigl\vert Y_a^{(n)}\bigr\vert^2 = 1,
  \quad \text{and} \quad \E \bigl(Y_a^{(n)}\bigr)^2 = 0
  \]
  for every $a \in G^{(n)}$; and that \eqref{E:Lindeberg-condition}
  holds. Then $\mu^{(n)}$ converges, in mean and in probability, to
  $\gamma_\C$.
\end{thm}

\medskip

The special case of Theorem \ref{T:circular-law-uncorrelated} for
classical circulant matrices (that is, when the $G^{(n)}$ are cyclic
groups) was proved by the author in \cite{Meckes}.

\begin{thm}\label{T:semicircle-law-general}
  Let $\alpha \in [0,1]$, $\beta > 0$.  Suppose that for each
  $n$, $\bigl\{Y_a^{(n)} \mid a \in G^{(n)} \bigr\}$ are mean $0$ and
  independent except for the constraint $Y_{a^{-1}}^{(n)} =
  \overline{Y_a^{(n)}}$; that
  \[
  \E Y_a^{(n)} Y_b^{(n)} = \begin{cases}
    1 & \text{ if } a = b^{-1} \neq a^{-1}, \\
    \alpha & \text{ if } a = b \neq a^{-1}, \\
    \beta & \text{ if } a = b = a^{-1}, \\
    0 & \text{ otherwise,} \end{cases}
  \]
  for $a, b \in G^{(n)}$; and that \eqref{E:Lindeberg-condition}
  holds. Assume further that $\lim_{n\to \infty} p_2^{(n)} = p$
  exists.  Then $\mu^{(n)}$ converges, in mean and in probability, to
  \[
  (1-p) \Normal\bigl(0, 1 + p (\beta - \alpha - 1)\bigr) 
  + p \Normal\bigl(0, 1 + \alpha + p (\beta - \alpha - 1) \bigr).
  \]
  if $p<1$ and to
  \(
  \Normal\bigl(0, \beta \bigr)
  \)
  if $p=1$.

\end{thm}

\medskip

Observe that by Lagrange's theorem on orders of subgroups,
$1/p_2^{(n)}$ is an integer, which implies that if $p<1$ then in fact
$p\le 1/2$, and therefore the stated variances of the normal
distributions named above are indeed positive.

The most obvious (though not necessarily, as we shall see, the most
natural) special case of interest in Theorem
\ref{T:semicircle-law-general} is when the $Y_a^{(n)}$ are real and
i.i.d.\ (except for the symmetry constraint), so that $\alpha = \beta
= 1$.  In that case the limiting spectral distribution is the mixture
distribution
\begin{equation}\label{E:iid-limit}
(1-p) \Normal(0,1-p) + p \Normal(0,2-p).
\end{equation}

Two other special cases are suggested by considering the analogy with
the GOE and GUE.  The $G$-circulant analogue of the GOE, as discussed
in the previous section, would have real entries such that $\alpha =
1$ and $\beta = 2$, and thus the limiting spectral distribution
\begin{equation}\label{E:GOE-limit}
(1-p) \Normal(0,1) + p \Normal(0,2).
\end{equation}
The slightly simpler nature of this limiting distribution (note that
the parameter $p$ plays only one role in \eqref{E:GOE-limit}, as
opposed to two roles in \eqref{E:iid-limit}) reflects that a
``GOE-like'' normalization of entries is more natural than equal
variances.  However, this phenomenon is only evident when $0 < p < 1$.
In the classical case of Wigner matrices it is well known that in
order for the semicircle law to hold, no variance assumption need be
made on the diagonal entries of the matrix.  The situation described
above emphasizes that this is the case precisely because the number of
diagonal entries in a Wigner matrix is negligible.

Finally, when the second moments are the same as for the
``$G$-circulant GUE'' of Proposition \ref{T:GUE-eigenvalues}, then
$\alpha = 0$ and $\beta = 1$ and, as in Corollary \ref{T:GUE-limit},
the limiting spectral distribution is simply the standard real
Gaussian distribution, even regardless of the value of $p$. Thus for
$G$-circulant matrices, a constraint to be complex Hermitian appears
to be somehow more natural than a constraint to be real symmetric.  As
in Theorem \ref{T:circular-law-uncorrelated}, the assumption that
$p_2$ approaches a limit can even be removed in this situation.

\begin{thm}\label{T:semicircle-law-special}
  Suppose that for each $n$, $\bigl\{Y_a^{(n)} \mid a \in G^{(n)}
  \bigr\}$ are mean $0$ and independent except for the constraint
  $Y_{a^{-1}}^{(n)} = \overline{Y_a^{(n)}}$; that $\E \bigl\vert
  Y_a^{(n)}\bigr\vert^2 = 1$ for every $a \in G^{(n)}$; that $\E
  \bigl( Y_a^{(n)} \bigr)^2 = 0$ if $a \neq a^{-1}$; and that
  \eqref{E:Lindeberg-condition} holds. Then $\mu^{(n)}$ converges, in
  mean and in probability, to $\gamma_\R$.
\end{thm}

\medskip

The special case of Theorem \ref{T:semicircle-law-special} for
classical circulant matrices (with more restrictive assumptions on the
distributions of the matrix entries) was proved by Bose and Mitra in
\cite{BoMi}.

We will not attempt to deal thoroughly with the question of when the
convergence in probability in the results above can be strengthened to
almost sure convergence.  However, the following result gives some
sufficient conditions.  Each of the conditions stated automatically
implies the Lindeberg-type condition \eqref{E:Lindeberg-condition};
for the first part this follows from exponential tail decay which is
implied by a Poincar\'e inequality (see \cite[Corollary 3.2]{Ledoux}),
and for the other parts it is elementary.

\begin{thm}\label{T:as-convergence}
  In the setting of Theorem \ref{T:circular-law-correlated},
  \ref{T:circular-law-uncorrelated}, \ref{T:semicircle-law-general},
  or \ref{T:semicircle-law-special}, suppose in addition that
  $\abs{G^{(n)}} = \Omega(n^{\eps})$ for some $\eps > 0$ and that one
  of the following conditions holds:
  \begin{enumerate}
  \item There is a constant $K > 0$ such that for every $n$ and every
    $a \in G^{(n)}$, $Y_a^{(n)}$ satisfies a Poincar\'e inequality
    with constant $K$.  That is,
    \[
    \Var f\bigl(Y_a^{(n)}\bigr) 
    \le K \E \bigl\vert \nabla f\bigl(Y_a^{(n)}\bigr)\bigr\vert^2
    \]
    for every smooth $f : \R^2 \to \R$.
  \item There is a constant $K > 0$ such that $\bigl\vert
    Y_a^{(n)}\bigr\vert \le K$ a.s.\ for every $n$ and every $a \in
    G^{(n)}$.
  \item For some $\delta \in (0,1]$, $\sup_{n \in \N} \max_{a \in G^{(n)}} \E
    \bigl\vert Y_a^{(n)} \bigr\vert^{2 + \delta} < \infty$, and $\sum_{n = 1}^\infty
    \abs{G^{(n)}}^{-\delta/2} < \infty$.
  \item For some $\delta \in (0,1]$, $\sup_{n \in \N} \max_{a \in
      G^{(n)}} \E \bigl\vert Y_a^{(n)}\bigr\vert^{2 + \delta} <
    \infty$, and $p_2^{(n)} \to p > 0$.
  \end{enumerate}
  Then $\mu^{(n)}$ converges to the stated limit almost surely.
\end{thm}

\bigskip

We now turn to the proofs of our main results. Unsurprisingly,
generalizing the results of the last section to non-Gaussian matrix
entries is achieved by using an appropriate version of the central
limit theorem to show that the eigenvalues $\lambda_\chi$ are
approximately distributed like uncorrelated Gaussian random variables.
Even to prove asymptotic results, it is necessary here to apply some
\emph{quantitative} version of the central limit theorem, in order to
achieve suitably uniform control over the $\lambda_\chi$. The approach
taken here (and earlier in \cite{Meckes}) generalizes and extends the
method used by Bose and Mitra in \cite{BoMi}, which applied a
multivariate version of the Berry--Esseen theorem and thus required
the matrix entries to have uniformly bounded third moments.  Here a
quantitative, multivariate version of Lindeberg's theorem is applied.

If $f : \R^d \to \R$ is bounded and Lipschitz with Lipschitz constant
$\abs{f}_L$, its bounded Lipschitz norm may be defined by
\[
\norm{f}_{BL} = \max\{ \norm{f}_\infty, \abs{f}_L\}.
\]
The bounded Lipschitz distance between random vectors $X$ and $Y$
in $\R^d$ is defined by
\[
d_{BL}(X,Y) = \sup_{\norm{f}_{BL} \le 1} \abs{\E f(X) - \E f(Y)}.
\]
It is well known (see e.g.\ \cite[section 11.3]{Dudley}) that the
class of bounded Lipschitz functions is a convergence-determining
class. The subclass of compactly supported such functions is
furthermore separable with respect to the sup norm \cite[Corollary
11.2.5]{Dudley}.  Thus to show that a sequence $\nu^{(n)}$ of
probability measures on $\R^d$ converges weakly to $\nu$ in mean, in
probability, or almost surely, it suffices to show that for each
bounded Lipschitz function $f$, $\nu^{(n)}(f) \to \nu(f)$ in the same
sense.

The following is a special case of \cite[Theorem 18.1]{BhRa} (cf.\ the
proof of \cite[Corollary 18.2]{BhRa}).

\begin{prop}\label{T:Lindeberg}
  Suppose that $X_1, \dotsc, X_k$ are independent mean $0$ random
  vectors in $\R^d$ such that $\frac{1}{k} \sum_{j=1}^k \Cov(X_j) =
  I_d$. For $\eps > 0$ let
  \[
  \theta(\eps) = \frac{1}{k} \sum_{j=1}^k \E \left(\norm{X_j}^2
    \ind{\norm{X_j} > \eps \sqrt{k}} \right).
  \]
  Then
  \[
  d_{BL} \left(\frac{1}{\sqrt{k}} \sum_{j=1}^k X_j, Z\right) \le
    C_d \inf_{0 \le \eps \le 1} (\eps + \theta(\eps)),
  \]
  where $Z$ is a standard Gaussian random vector in $\R^d$, and $C_d >
  0$ depends only on $d$.
\end{prop}

\medskip

\begin{proof}[Proof of Theorem \ref{T:circular-law-correlated}]
  Let $f:\C \to \R$ with $\norm{f}_{BL} \le 1$.  Observe that
  \begin{equation}\label{E:mean-int}
    \E \mu(f) = 
    \frac{1}{\abs{G}} \sum_{\chi \in \widehat{G}} \E f(\lambda_\chi)
    = \frac{1}{\abs{G}} \sum_{\chi \in \widehat{G}} 
    \E f\left(\frac{1}{\sqrt{\abs{G}}}\sum_{a \in G} \chi(a) Y_a\right),
  \end{equation}
  where $(n)$ superscripts have been omitted for simplicity.  We
  consider $\lambda_\chi$ as a sum of independent random vectors in
  $\R^2 \cong \C$. The relevant covariances are
  \[
  \Cov(\chi(a) Y_a) = \begin{bmatrix}
    \E (\Re \chi(a) Y_a)^2 & \E (\Re \chi(a) Y_a) (\Im \chi(a) Y_a) \\
    \E (\Re \chi(a) Y_a) (\Im \chi(a) Y_a) & \E (\Im \chi(a) Y_a)^2
  \end{bmatrix}.
  \]
  The identities
  \begin{equation}\begin{split}\label{E:Re-Im}
      (\Re w)(\Re z)
      &= \tfrac{1}{2} \Re \bigl[(w + \overline{w}) z \bigr], \\
      (\Im w)(\Im z)
      &= \tfrac{1}{2} \Re \bigl[(\overline{w} - w) z \bigr], \\
      (\Re w)(\Im z) &= \tfrac{1}{2} \Im \bigl[(w - \overline{w}) z
      \bigr],
  \end{split}\end{equation}
  will be useful.
  
  Setting $w = z = \chi(a) Y_a$ for a fixed $\chi \in \widehat{G}$,
  \begin{align*}
    \sum_{a \in G} \E(\Re \chi(a) Y_a)^2 &=
    \sum_{a \in G} \left[\frac{1}{2} \Re \E 
      \left(\chi(a)^2 Y_a^2 + \abs{\chi(a)}^2 \abs{Y_a}^2 \right)\right] \\
    &= \frac{1}{2}\left( \abs{G} + \alpha \sum_{a \in G} \chi^2(a) \right)
    = \frac{\abs{G}}{2}\bigl( 1 + \alpha \ind{\chi = \overline{\chi}}\bigr).
  \end{align*}
  In the last step we have used that unless $\chi$ is real-valued, $\chi$
  and $\overline{\chi}$ are distinct characters, and hence orthogonal in
  $\ell^2(G)$. In similar fashion, we find that
  \[
  \Cov\bigl(\lambda_\chi\bigr) 
  = \frac{1}{\abs{G}} \sum_{a \in G}
  \Cov \bigl(\chi(a) Y_a\bigr) = \frac{1}{2} \bigl( I_2 +
  \ind{\chi = \overline{\chi}} \bigl[\begin{smallmatrix} \alpha & 0 \\
    0 & - \alpha \end{smallmatrix}\bigr]\bigr).
  \]
  Observe in particular that if $\alpha = 1$ and $\chi$ is
  real-valued, then $\lambda_\chi$ is almost surely real, with
  variance $1$; in that case we treat $\lambda_\chi$ as a random
  variable in $\R$, as opposed to a random vector in
  $\R^2$. Proposition \ref{T:Lindeberg} and
  \eqref{E:Lindeberg-condition} (recalling that $\abs{\chi(a)} = 1$
  always) now imply that there is a sequence $\delta_n$ decreasing to
  $0$ such that for each $\chi \in \widehat{G}$,
  \[
  \abs{\E f\left(\frac{1}{\sqrt{\abs{G}}}\sum_{a \in G} \chi(a) Y_a\right)
    - \gamma_\alpha(f)} \le \delta_n
  \]
  if $\chi$ is real-valued, and 
  \[
  \abs{\E f\left(\frac{1}{\sqrt{\abs{G}}}\sum_{a \in G} \chi(a) Y_a\right)
    - \gamma_\C(f)} \le \delta_n
  \]
  otherwise. Writing $\nu^{(n)} = (1-p_2^{(n)}) \gamma_\C + p_2^{(n)}
  \gamma_\alpha$, by \eqref{E:mean-int} it follows that
  \begin{equation}\label{E:mean-bound} \begin{split}
      \abs{\E \mu(f) - \nu(f)} & = \Biggl\vert \frac{1}{\abs{G}}
      \sum_{\chi = \overline{\chi}} \E f\left(\frac{1}{\sqrt{\abs{G}}}
        \sum_{a \in G} \chi(a) Y_a\right)  - p_2
        \gamma_\alpha(f) \\
      & \qquad + \frac{1}{\abs{G}}\sum_{\chi
          \neq \overline{\chi}} \E f\left(\frac{1}{\sqrt{\abs{G}}} \sum_{a
            \in G}
          \chi(a) Y_a\right) - (1-p_2) \gamma_\C(f) \Biggr\vert \\
      & \le p_2 \delta_n + (1-p_2) \delta_n = \delta_n,
    \end{split}
  \end{equation}
  where as above the subscripts $(n)$ are omitted.  Since $p_2^{(n)}
  \to p$, it follows that $\nu^{(n)} \Rightarrow (1-p) \gamma_\C + p
  \gamma_\alpha$, and so $\E \mu^{(n)} \Rightarrow (1-p) \gamma_\C + p
  \gamma_\alpha$.
  
  \medskip

  Next observe that
  \begin{equation}\label{E:var-int}
    \E \bigl(\mu(f) \bigr)^2 =
    \frac{1}{\abs{G}^2} \sum_{\chi_1, \chi_2 \in \widehat{G}}
    \E f(\lambda_{\chi_1}) f (\lambda_{\chi_2})
    = \frac{1}{\abs{G}^2} \sum_{\chi_1, \chi_2 \in \widehat{G}}
    \E F\bigl( (\lambda_{\chi_1}, \lambda_{\chi_2}) \bigr),
  \end{equation}
  where $F:\C^2 \to \R$ is defined by $F(w,z) = f(w)f(z)$, so that
  $\norm{F}_{BL} \le 2$. We now consider
  $(\lambda_{\chi_1},\lambda_{\chi_2})$ as a sum of independent random
  vectors in $\R^4$. The upper-left and lower-right $2\times 2$ blocks
  of $\Cov \bigl((\lambda_{\chi_1}, \lambda_{\chi_2}) \bigr)$ are of
  course just $\Cov(\lambda_{\chi_1})$ and $\Cov(\lambda_{\chi_2})$,
  computed above. For the off-diagonal blocks, we use $w = \chi_1(a)
  Y_a$ and $z = \chi_2(a) Y_a$ in \eqref{E:Re-Im} to obtain for
  example
  \begin{align*}
    \sum_{a \in G} \E(\Re \chi_1(a) Y_a) (\Re \chi_2(a) Y_a) &=
    \sum_{a \in G} \left[\frac{1}{2}\Re \E \left( \chi_1(a) \chi_2(a)
        Y_a^2 + \overline{\chi_1(a)} \chi_2(a)
        \abs{Y_a}^2 \right) \right] \\
    &= \frac{1}{2} \left( \alpha \sum_{a \in G} \chi_1(a) \chi_2(a)
      + \sum_{a \in G} \overline{\chi_1(a)} \chi_2(a)\right) \\
    &= \frac{\abs{G}}{2} \bigl( \alpha \ind{\chi_1 =
      \overline{\chi_2}} + \ind{\chi_1 = \chi_2}\bigr).
  \end{align*}
  Similarly, it follows that the off-diagonal blocks of
  $\Cov\bigl((\lambda_{\chi_1}, \lambda_{\chi_2})\bigr)$ are $0$
  unless $\chi_1 = \chi_2$ or $\chi_1 = \overline{\chi_2}$.

  Assume for now that $\chi_1 \neq \chi_2$ and $\chi_1 \neq
  \overline{\chi_2}$. Applying Proposition \ref{T:Lindeberg}, we now
  obtain that there is a sequence $\delta'_n$ decreasing to $0$ such
  that whenever $\norm{f}_{BL} \le 1$,
  \[
  \abs{\E f(\lambda_{\chi_1}) f(\lambda_{\chi_2}) 
    - \gamma_\alpha(f)^2} \le \delta'_n
  \]
  if $\chi_1$ and $\chi_2$ are both real-valued,
  \[
  \abs{\E f(\lambda_{\chi_1}) f(\lambda_{\chi_2}) 
    - \gamma_\C(f) \gamma_\alpha(f)} \le \delta'_n
  \]
  if exactly one of $\chi_1$ and $\chi_2$ is real-valued, and
  \[
  \abs{\E f(\lambda_{\chi_1}) f(\lambda_{\chi_2}) 
    - \gamma_\C(f)^2} \le \delta'_n
  \]
  if neither $\chi_1$ nor $\chi_2$ is real-valued.  (Note that
  Proposition \ref{T:Lindeberg} may be applied in the case of
  nonidentity covariance via a linear change of coordinates. For
  $\alpha < 1$, the determinant of the covariance is bounded away from
  zero, whereas for $\alpha = 1$ the variables are real.) Given
  $\chi_1$, note that there are at most $2$ characters $\chi_2$ which
  are unaccounted for.  By \eqref{E:var-int}, it now follows that
  \begin{equation}\label{E:var-bound}
  \abs{\E \mu^{(n)}(f)^2 
   - \nu^{(n)}(f)^2} 
   \le \delta'_n + \frac{2}{\abs{G^{(n)}}}.
  \end{equation}

  Finally,
  \begin{align*}
    \E \abs{\mu^{(n)}(f) - \nu^{(n)}(f)}^2 
    & = \bigl[\E \mu^{(n)}(f)^2 - \nu^{(n)}(f)^2\bigr] 
    - 2 \nu^{(n)}(f) \bigl[\E \mu^{(n)}(f) - \nu^{(n)}(f)\bigr] \\
    & \le \abs{\E \mu^{(n)}(f)^2 - \nu^{(n)}(f)^2}
     + 2 \abs{\E \mu^{(n)}(f) - \nu^{(n)}(f)},
  \end{align*}
  so by \eqref{E:mean-bound} and \eqref{E:var-bound},
  \[
  \mu^{(n)} (f) \to \bigl[(1-p) \gamma_\C + p \gamma_\alpha\bigr](f)
  \]
  in $L^2$, and hence in probability.
\end{proof}

\medskip

\begin{proof}[Proof of Theorem \ref{T:circular-law-uncorrelated}]
  The proof is analogous to that of Theorem
  \ref{T:circular-law-correlated}, setting $\alpha = 0$. In that case
  $\Cov(\lambda_\chi)$ no longer depends on whether $\chi$ is
  real-valued, which makes it unnecessary to assume that $p_2^{(n)}$
  approaches a limit.
\end{proof}

\medskip

\begin{proof}[Proof of Theorem \ref{T:semicircle-law-general}]
  We omit $(n)$ superscripts as before.  We will assume that $p<1$;
  the case $p=1$ (which implies that in fact $p_2 = 1$ for
  sufficiently large $n$) is similar and slightly simpler.  Let $A =
  \{ a \in G \mid a = a^{-1} \}$.  Since $G$ is abelian, $A$ is a
  subgroup of $G$.  The restriction of a character of $G$ to $A$ is a
  character on $A$, which is necessarily real-valued on $A$.  It
  follows that for $\chi_1, \chi_2 \in \widehat{G}$,
  \begin{equation}\label{E:semicircle-covariance}\begin{split}
    \abs{G} \E \lambda_{\chi_1} \lambda_{\chi_2} &=
    \sum_{a,b \in G} \chi_1(a) \chi_2(b) \E Y_a Y_b \\
    & = \sum_{a \in G} \chi_1(a) \bigl[\bigl(\overline{\chi_2(a)} +
    \alpha \chi_2(a)\bigr) \ind{a \neq a^{-1}}
    + \beta \chi_2(a) \ind{a = a^{-1}} \bigr] \\
    & = \sum_{a \in G \setminus A} \chi_1(a) \overline{\chi_2(a)} + \alpha
    \sum_{a \in G \setminus A} \chi_1(a) \chi_2(a)
    + \beta \sum_{a \in A} \chi_1(a) \chi_2(a) \\
    & = \sum_{a \in G} \chi_1(a) \overline{\chi_2(a)} + \alpha \sum_{a
      \in G} \chi_1(a) \chi_2(a)
    + (\beta - \alpha - 1) \sum_{a \in A}  \chi_1(a) \chi_2(a) \\
    & = \abs{G} \bigl( \ind{\chi_1 = \chi_2} + \alpha \ind{\chi_1 =
      \overline{\chi_2}}\bigr) +  \abs{A} (\beta - \alpha - 1) \ind{\chi_1
      \vert_A = \overline{\chi_2} \vert_A}  \\
    & = \abs{G} \bigl( \ind{\chi_1 = \chi_2} + \alpha \ind{\chi_1 =
      \overline{\chi_2}} + p_2 (\beta - \alpha - 1) \ind{\chi_1
      \vert_A = \chi_2 \vert_A} \bigr).
  \end{split}\end{equation}
  In particular, for $\chi \in \widehat{G}$,
  \[
  \Var(\lambda_\chi) = 1 + \alpha \ind{\chi = \overline{\chi}}
  + p_2 (\beta - \alpha - 1).
  \]
  Denoting
  \[
  \nu^{(n)} = (1-p_2^{(n)}) \Normal\bigl(0,1 + p_2^{(n)}(\beta -
  \alpha - 1)\bigr) + p_2^{(n)} \Normal\bigl(0, 1 + \alpha + p_2^{(n)}(\beta - 
  \alpha - 1)\bigr),
  \]
  it follows as in the proof of Theorem
  \ref{T:circular-law-correlated} that $d_{BL}(\E \mu^{(n)},
  \nu^{(n)}) \to 0$, and thus that 
  \[
  \E \mu^{(n)} \Rightarrow (1-p) \Normal\bigl(0,1 + p(\beta - \alpha -
  1)\bigr) + p \Normal\bigl(0, 1 + \alpha + p(\beta - \alpha -
  1)\bigr).
  \]
  In this situation just the $1$-dimensional case of Proposition
  \ref{T:Lindeberg} is necessary.  Observe also that the variances
  $\Var(\lambda_\chi)$ are uniformly bounded away from $0$ (cf.\ the
  comments following the statement of the theorem.) This is necessary
  so that Proposition \ref{T:Lindeberg} may be applied for nonidentity
  covariance, via a linear change of coordinates, and still yield
  error bounds $\delta_n$ which are uniform in $f$ with $\norm{f}_{BL}
  \le 1$.
  
  \medskip

  By \eqref{E:semicircle-covariance}, if $\chi_1 \neq \chi_2$ and
  $\chi_1 \neq \overline{\chi_2}$, then
  \begin{align*}
  \Cov \bigl((\lambda_{\chi_1}, \lambda_{\chi_2})\bigr)
  & = \bigl(1 + p_2 (\beta - \alpha - 1)\bigr) I_2 + \alpha 
  \begin{bmatrix} \ind{\chi_1 = \overline{\chi_1}} & 0 \\
    0 & \ind{\chi_2 = \overline{\chi_2}} \end{bmatrix} \\
  & \quad + p_2 (\beta -
  \alpha - 1) \ind{\chi_1\vert_A = \chi_2\vert_A}
  \begin{bmatrix} 0 & 1 \\ 1 & 0 \end{bmatrix}. 
  \end{align*}

  We consider separately the cases $p = 0$ and $p > 0$.  If $p = 0$,
  then when $\chi_1 \neq \chi_2$ and $\chi_1 \neq \overline{\chi_2}$,
  we have
  \[
  \Cov \bigl((\lambda_{\chi_1}, \lambda_{\chi_2})\bigr)
  = \begin{bmatrix} 1 + \alpha \ind{\chi_1 = \overline{\chi_1}} & 0 \\
    0 & 1 + \alpha \ind{\chi_2 = \overline{\chi_2}} \end{bmatrix}
  + o(1).
  \]
  From here the argument is completed as in the proof of Theorem
  \ref{T:circular-law-correlated}.

  Suppose now that $p > 0$. Given $\chi_1 \in \widehat{G}$, by Lemma
  \ref{T:extensions} there are exactly $\frac{1}{p_2}$ values of
  $\chi_2 \in \widehat{G}$ with $\chi_1\vert_A = \chi_2\vert_A$.
  Therefore,
  \[
  \Cov \bigl((\lambda_{\chi_1}, \lambda_{\chi_2})\bigr) = \bigl(1 +
  p_2 (\beta - \alpha - 1)\bigr) I_2 + \alpha
  \begin{bmatrix} \ind{\chi_1 = \overline{\chi_1}} & 0 \\
    0 & \ind{\chi_2 = \overline{\chi_2}} \end{bmatrix}
  \]
  for all but a negligible fraction of pairs $\chi_1, \chi_2 \in
  \widehat{G}$.  The argument is again completed as in the proof of
  Theorem \ref{T:circular-law-correlated}.
\end{proof}

\medskip

\begin{proof}[Proof of Theorem \ref{T:semicircle-law-special}]
  The proof is analogous to that of Theorem
  \ref{T:semicircle-law-general}, setting $\alpha = 0$ and $\beta =
  1$. In that case $\Cov(\lambda_{\chi_1}, \lambda_{\chi_2}) =
  \ind{\chi_1 = \chi_2}$, so it is unnecessary to assume that
  $p_2^{(n)}$ approaches a limit.
\end{proof}

\medskip

\begin{proof}[Proof of Theorem \ref{T:as-convergence}]
  \begin{enumerate}
  \item The Poincar\'e inequality assumption and independence imply an
    exponential concentration property for the family of eigenvalues
    $\bigl\{\lambda_\chi \mid \chi \in \widehat{G}^{(n)}\bigr\}$.  In
    particular, combining Corollaries 5.7 and 3.2 of \cite{Ledoux}, it
    follows that for each $L$-Lipschitz $F: \ell^2(G^{(n)}) \to \R$,
    \[
    \Prob\left[\abs{F\bigl(Y^{(n)}\bigr) - \E F\bigl(Y^{(n)}\bigr)}
      \ge t \right] \le 2 e^{-c t /\sqrt{K} L}
    \]
    for each $t > 0$, where $c > 0$ is some absolute constant and
    $Y^{(n)}$ is shorthand for $\bigl(Y_a^{(n)}\bigr)_{a\in G^{(n)}}$.
    Now for a $1$-Lipschitz $f:\C \to \R$ and $k \in \N$,
    \[
    \abs{\frac{1}{k} \sum_{j=1}^k f(w_j) - \frac{1}{k} \sum_{j=1}^k f(z_j)}
    \le \frac{1}{k} \sum_{j=1}^k \abs{w_j - z_j}
    \le \sqrt{\frac{1}{k} \sum_{j=1}^k \abs{w_j - z_j}^2}
    \]
    by the Cauchy--Schwarz inequality. Combining this with Lemma
    \ref{T:FT-isometry}(\ref{I:isometry}) it follows that
    $\mu^{(n)}(f)$ is $\abs{G^{(n)}}^{-1/2}$-Lipschitz as a function of
    $Y^{(n)}$, and so
    \[
    \Prob\left[\abs{\mu^{(n)}(f) - \E \mu^{(n)}(f)\bigr)}
      \ge t \right] \le 2 e^{-c t \sqrt{\abs{G^{(n)}}/ K}}.
    \]    
    Combined with the already known convergence in mean and the
    Borel--Cantelli lemma, this implies almost sure convergence of
    $\mu(f)$.
  \item The proof is similar to the previous part, using instead
    Talagrand's convex-distance concentration inequality for
    independent bounded random variables \cite[Theorem
    4.1.1]{Talagrand-ihes} (see e.g.\ \cite[Corollary 4]{Meckes-jfa}
    for an explicit statement of a version that applies directly to
    complex random variables), cf.\ the proof of \cite[Theorem
    2]{Meckes}).
  \item The stated Lyapunov-type assumption yields upper bounds on all
    the $\delta_n$ quantities in the proofs above of order
    $\abs{G^{(n)}}^{-\delta/2}$ for $0 < \delta \le 1$ (cf.\
    \cite[Corollary 18.3]{BhRa}). Thus the assumption that
    $\sum_{n=1}^\infty \abs{G^{(n)}}^{-\delta/2}$ allows the
    Borel--Cantelli lemma to be applied again.
  \item The assumption that $p > 0$ implies that $\abs{G^{(n)}}$
    actually grows exponentially: since $p_2^{(n)}$ is always the
    reciprocal of an integer (by Lagrange's theorem about the orders
    of subgroups of finite groups), $p_2^{(n)} \to p > 0$ implies that
    $p_2^{(n)}$ is eventually constant.  By the classification of
    finite abelian groups,
    \[
    G \cong \left(\prod_{j = 1}^m \Z_{2^{k_j}} \right) \times H,
    \]
    where $m \ge 0$, $k_j \ge 1$ for each $j$, and each nonidentity
    element of $H$ has odd order.  (For simplicity of notation, we are
    again suppressing the dependence of all these on $n$.)  In this
    notation, the number of $a \in G$ such that $a = a^{-1}$ is $2^m$,
    so that $\abs{G} = 2^m / p_2$.  The hypothesis that
    $\abs{G^{(n)}}$ is strictly increasing thus implies that $m$ is
    eventually strictly increasing, and hence $\abs{G^{(n)}}$ is
    eventually \emph{exponentially} increasing.  Therefore the
    previous part of the theorem applies. \qedhere
  \end{enumerate}
\end{proof}

\bibliographystyle{plain}
\bibliography{G-circ-abelian}

\end{document}